\documentclass[11pt,a4paper]{article}
\usepackage{amsmath,amssymb,amsfonts,amsthm}

\newcommand{\E}{{\bf{E}}}
\newcommand{\PP}{{\bf{P}}}
\newcommand{\Var}{{\bf{Var}}}

\newtheorem{tm}{Theorem}

\theoremstyle{definition}

\begin{document}
\bibliographystyle{plain}
\parindent=0pt
\centerline{\LARGE \bfseries Birth of a strongly connected giant}
\centerline{\LARGE \bfseries in an  inhomogeneous random digraph }

{\large  \qquad\qquad\qquad \qquad \qquad}\\

\par\vskip 3.5em

\centerline{ M. Bloznelis \footnote{Research supported by SFB 701 at the University of Bielefeld, Germany}, \  F. G\"otze, \  J. Jaworski\footnote{Research supported by the Marie Curie Intra-European Fellowship No. 236845 (RANDOMAPP) within the 7th EU Framework Programme}}

\bigskip

\bigskip

 Vilnius University,  \qquad      \qquad        Bielefeld University    \qquad    \qquad  Adam Mickiewicz University

 LT-03225 Vilnius     \qquad      \qquad     \ \    D-33501 Bielefeld    \qquad      \qquad \ \   61--614 Pozna\'n

 Lithuania     \qquad      \qquad   \qquad    \qquad   Germany      \qquad  \qquad       \qquad   \qquad  \   Poland

\par\vskip 2.5em

\begin{abstract}
We present and investigate a general model  {for} inhomogeneous random
digraphs  {with labeled vertices}, where  {the} arcs  {are generated} independently, and the
probability of  {inserting}  an arc depends
on  {the labels} of its endpoints and  {its} orientation.
For this model the critical point  {for the emergence of a giant component}
is determined  {via a} branching  {process} approach.
\par\end{abstract}

\smallskip
{\bfseries key words:}  inhomogeneous digraph, phase transition, giant component.

\par\vskip 5.5em
\section{Introduction}
Random directed graphs (digraphs) are widely used for modeling
networks arising e.g. in physics, biology,  social studies and,  more recently, bioinformatics, linguistics and
 {the analysis of networks in the internet}.

It has been observed that the well studied  classical homogeneous random digraph model $D(n,p)$, where
arcs are inserted independently and with the same probability $p$,
may not fit  real life networks,  because the latter  often exhibit statistical properties, such as, e.g.,
power law in-degree/out-degree
distribution, that are inconsistent with the model $D(n,p)$. Generally, the real life networks are inhomogeneous
(see, e.g., Albert and Barab\'asi \cite{Albert}, Dorogovtsev, Mendes and Samukhin \cite{DorogovtsevMS},
Durrett \cite{Durrett2007}, Newman, Strogatz and Watts \cite{NSW2001}).

In this paper we study a very general model of sparse inhomogeneous random digraphs with {\it independent} arcs.
By 'sparse'
we mean that the number of arcs does not grow faster
 than linearly in $n$, where $n$ is the number
of particles (called vertices of the digraph). By inhomogeneity we mean that different arcs
  {are inserted} with different probabilities. In particular, our random digraph
  model is able to produce
a wide class of asymptotic in-degree  and out-degree distributions  including power law distributions.

The main question addressed in the paper is the description of the phase
transition in  {the} (strongly connected)
cluster size, that is, the
emergence of a giant strongly connected component.
To answer this question we establish  a first order asymptotic for the number $N_1$ of vertices
in the largest strongly connected component, by showing $N_1=\rho n+o_P(n)$.
The fraction  $\rho\ge 0$  is expressed in terms of  survival probabilities
of related branching processes that reflect the statistical properties  of  a
certain neighbourhood of a  randomly chosen vertex.


 Let us mention that mathematically rigorous results on phase transition in  $D(n,p)$ were
 established in  Karp \cite{Karp} and  {\L}uczak \cite{Luczak1990},
 see also
{\L}uczak  and Seierstad \cite{LuczakSd2009}.   {More general  three
  parameter models were}
studied in {\L}uczak and Cohen \cite{LuczakC1992}.
  The phase transition  in a random digraph with given (non-random)
in-degree and out-degree  sequences was shown in Cooper and Frieze \cite{Cooper-Frieze2004}. We  also want to mention the related in-depth
study of the phase
transition phenomenon in general inhomogeneous random graphs  {in} Bollob\'as, Janson and Riordan \cite{BJR}, which has inspired  our
work on inhomogeneous digraphs.
\bigskip

The paper is organized as follows. In Section 2 we define  {a}  finite dimensional model of a random inhomogeneous
digraph and for this model determine the critical point of the phase transition. Section 3 extends finite
dimensional results to  inhomogeneous digraphs defined by very general (possibly infinite dimensional)
kernels.
Proofs are postponed to Section 4.

\section{The finite dimensional model}

Before presenting our results we briefly recall some relevant facts and notation related to digraphs.

{\bf 2.1.} A digraph $D$ on the vertex set $V=\{v_1,\dots, v_n\}$ is a subset of the set $[V]^2=\{(u,v),u,v\in~V\}$ of all ordered pairs of elements of $V$.  Elements of $D$ are called directed edges or arcs. The fact that $(u,v)$ is an element of $D$ is denoted  $\{u \rightarrow v\}\in D$ or just $u \rightarrow v$.
More generally, if for some distinct vertices $w_1,\dots w_k\in V$ the collection of arcs
${\cal P}=\{(w_1,w_2)$, $(w_2, w_3)$, $\, \dots,$ $(w_{k-1}, w_k)\}$ is a subset of   $D$ then  ${\cal P}$ is called
a directed path ({\it d-path}) starting at $w_1$ and ending at $w_k$, and
denoted  {by}  $w_1\rightsquigarrow w_k$.
If $v\rightsquigarrow w$ and $w\rightsquigarrow v$  then  {the}
vertices $v$ and $w$ are said to {\it communicate}. In this case  we write $v\leftrightsquigarrow w$. In addition,
we  {define}  
that $v\leftrightsquigarrow v$, for every $v\in V$, even in the case where
the loop $v\rightsquigarrow v$ is not present in $D$.
 A digraph is called
{\it strongly connected} if every pair of its vertices  {communicates}.
%
 {Since}  $'\leftrightsquigarrow '$ is an equivalence relation, it  splits the vertex set
$V$ into a union of disjoint subsets 
 {of elements communicating}  with each other.
 The subgraph of $D$ induced by such a subset of vertices is  called a
 strongly connected component ({\it SC-component}).

 We are interested in the fast growth of the largest SC-component when the
 density of random arcs gradually increases  in the range $\Theta(n^{-1})$.
  Here the important parameter is  the size  $N_1=N_1(D)$ of the largest SC-component
 (the number of vertices of the SC-component, which has the largest number of vertices).
Another interesting parameter is  the size  $N_2=N_2(D)$ of the second largest SC-component.

{\bf 2.2.} We assume that vertices 
 {belong to different types}  and  {that} the probability of an arc
depends   on the types of its endpoints and the scale parameter $n$
 {only}.
In addition,  we assume
 throughout this section
 that the set of different types is finite and as the number of  {vertices} increases.
 {A~similar}  model of {\it random graphs} (but not digraphs) has been  introduced  {by} S\"oderberg \cite{soderberg}.

Let us introduce some  {more} notation. Let $S=\{s_1, s_2, \dots, s_k\}$ denote the set of types,
 and  let $s(v)$   denote the type   of
 a vertex $v\in V$.
We write
$n=n_1+n_2+\dots+n_k$,
where each $n_i=\# \{v\in~V:~s(v)=s_i\}$ 
 {denotes the number of}  vertices of type $s_i\in S$.

Given a integer vector ${\overline n}=(n_1,\dots, n_k)$ and a
 $k\times k$ matrix ${\mathbb P}=||p_{ij}||$ with non-negative entries, define the inhomogeneous
  random digraph ${\cal D}$ on the vertex set $V$ as follows. The set of arcs of ${\cal D}$ is drawn at random from $[V]^2$
  so that events $\{ u \rightarrow v\}\in{\cal D}$ are independent and have probabilities
  $\PP(u \rightarrow v) =  1\wedge(p_{ij} n^{-1})$, for each $(u,v)\in [V]^2$. Here $i$ and $j$ refer to the types
  $s_i=s(u)$ and $s_j=s(v)$ of the endpoints $u$ and $v$.
  We use the  notation $a\wedge b$ and $a\vee b$ for
$\min\{a,b\}$ and
$\max\{a,b\}$ respectively.
  In order to stress the dependence of the model on the parameters
${\mathbb P}$ and ${\overline n}$ we sometimes write ${\cal D}={\cal D}_{{\mathbb P}, {\overline n}}$.

  We shall assume that the fraction of vertices of a given type is asymptotically constant. That is, there is a probability distribution $Q$ on the type space $S$ such that for each $q_i=Q(s_i)$ we have
\begin{equation}\label{Q-n}
q_i>0
\qquad
{\text{and}}
\qquad
n_i-q_in=o(n)
\qquad
{\text{as}}
\qquad
n\to\infty.
\end{equation}

Clearly, for large $n$, the typical/statistical characteristics of the digraph ${\cal D}_{{\mathbb P},{\overline n}}$ depend solely on the distribution $Q$ and matrix ${\mathbb P}=||p_{ij}||$.  In order to describe the phase transition
in ${\cal D}_{{\mathbb P},{\overline n}}$ in terms of $Q$ and  ${\mathbb P}$ we use the "language" of branching processes.

 Let us  consider
 multi-type Galton-Watson processes where particles are of  types from
 $S$.
 Given $s\in S$, let  ${\cal X}(s)$ (respectively ${\cal Y}(s)$) denote the Galton-Watson process starting
 at a particle of type $s$ such that the number of children of type $s_j\in S$ of
 a particle of type $s_i\in S$ has Poisson distribution with
 mean
$p_{ij}q_j$ (respectively $p_{ji}q_j$), $1\le i,j\le k$.
We write ${\cal X}=\{{\cal X}(s),\, s\in S\}$ and ${\cal Y}=\{{\cal Y}(s),\, s\in S\}$.
Let $\rho_{\cal X}(s)$ and $\rho_{\cal Y}(s)$ denote the non-extinction probability of
 ${\cal X}(s)$ and ${\cal Y}(s)$ respectively.
Write
\begin{equation}\label{RO-1}
\rho=\rho_{{\cal X}{\cal Y}}=\sum_{1\le i\le k}\rho_{\cal X}(s_i)\rho_{\cal Y}(s_i)q_i.
\end{equation}

We show in Theorem 1 below that a giant SC-component of range $n$ emerges in ${\cal D}$ whenever   $\rho$ is positive.

\begin{tm}\label{T1} Assume that (\ref{Q-n}) holds.  Assume that the matrix ${\mathbb P}=||p_{ij}||$ is irreducible.
 As $n\to\infty$ we have
\begin{equation}\label{result}
N_1({\cal D}_{{\mathbb P},{\overline n}})=\rho_{{\cal X}{\cal Y}}\, n+o_P(n)
\end{equation}
and $N_2=o_P(n)$.
\end{tm}

\medskip
Here for a sequence of random variables $\{ Z_n\}$ we write $Z_n=o_P(n)$ if $\lim_n\PP(|Z_n|>\delta n)=0$ for
each $\delta>0$.

Recall that the matrix $||p_{ij}||$ is called {\it irreducible} if its associated  digraph $D_P$ is strongly connected. Here $D_P$ is the digraph on the vertex set $S$ such that $\{s_i\rightarrow s_j\}\in D_P$ whenever $p_{ij}>0$.

{\it Remark.} In the case where $D_P$ is not strongly connected, Theorem 1 can be applied to its strongly connected components.
Let $D_{1}$,\dots, $D_{r}$ denote
the strongly connected components of $D_P$. Given $1\le m\le r$, let
 $S_m\subset S$ denote the vertex set of $D_{m}$ and
let ${\cal D}_m$ denote the subgraph of the
random digraph ${\cal D}$ induced  by the vertex set $V_m=\{v\in V:\, s(v)\in S_m\}$.
  Note that, for $i\not=j$, any two vertices  $v\in V_i$ and
 $u\in V_j$ do
 not communicate in ${\cal D}$. Therefore, each strongly connected component of ${\cal D}$ is a subgraph of some
 ${\cal D}_m$.
 The asymptotic size of the largest strongly connected component of ${\cal D}_m$ is, by Theorem 1,
 $N_1({\cal D}_m)=\rho_m n+o_P(n)$. Here we denote $\rho_m=\sum_{s\in S_m}\rho_{\cal X}(s)\rho_{\cal Y}(s)Q(s)$.
 It follows now that $N_1({\cal D})=\max_{1\le m\le r}\rho_m n+o_P(n)$.
 In this way we obtain an extension of (\ref{result})
 to general (not necessarily strongly connected) digraphs $D_P$.

\section{The general model}
The result of Theorem \ref{T1} extends to a much more general situation where the type space $S$ is~a~separable metric space,
$Q$ is a probability measure defined on Borel sets of $S$. Here the matrix $||p_{ij}||$ is replaced by
a non-negative $S\times S$ measurable kernel $\kappa(s,t)$, $s,t\in S$.

Let $x_1,x_2,\dots$ be a sequence of random variables with values in $S$  such that the empirical distribution
of the first $n$ observations $x_1,\dots, x_n$
 approximates the measure $Q$ in proba\-bi\-lity as $n\to\infty$. That is, we assume that
 for each $Q$-continuous Borel set
$A\subset S$ we have
$\#\{i\in~[1,n]:~x_i\in A\}n^{-1}=Q(A)+o_P(1)$ as $n\to\infty$. Recall that a Borel set $A$ is
called {\it $Q$-continuous}
whenever its boundary $\partial A$ has zero probability $Q(\partial A)=0$.

Given $n$, let ${\cal D}_n$ be the random digraph on the vertex set $\{x\}_1^n=\{x_1,\dots, x_n\}$ with
independent arcs
having probabilities $\PP(\{x_i\rightarrow x_j\}\in  {\cal D}_n)=1\wedge(n^{-1} \kappa(x_i,x_j))$,
$1\le i,j\le n$. Combining $S$, $Q$ and $\kappa$ we obtain a very large class of inhomogeneous digraphs
with independent arcs. Obviously, the model
 will include digraphs with in-degree and out-degree distributions which have power laws.

Such a general model, for {\it random graphs} (not digraphs), was introduced in Bollob\'as, Janson and Riordan \cite{BJR}.
Note that in the case of random
graphs it is necessary to assume, in addition, that the kernel $\kappa$ is symmetric. In the definition of digraphs
 ${\cal D}_n$, $n\ge 2$, we do not require
the symmetry of the kernel.

For large $n$, the phase transition in the digraph ${\cal D}_n$ can be described in terms of the survival
probabilities of the related multi-type Galton--Watson branching processes with type space $S$.
 Given $s\in S$, let  ${\cal X}(s)$ (respectively ${\cal Y}(s)$) denote the Galton-Watson process starting
 at~a~particle of type $s$ such that the number of children of types in a subset $A\subset S$
  of
 a particle of type $t\in S$ has Poisson distribution with
 mean
$\int_{A}\kappa(t,u)Q(du)$ (respectively $\int_{A}\kappa(u,t)Q(du)$). These numbers are independent
for disjoint subsets $A$ and for different particles.
 The critical point of the emergence of the giant SC-component is determined  by the averaged joint survival probability
\begin{equation}\label{RO-W2}
\rho_{{\cal X}{\cal Y}}=\int_S\rho_{\cal X}(s)\rho_{\cal Y}(s)Q(ds)
\end{equation}
being positive.
Here $\rho_{\cal X}(s)$ and $\rho_{\cal Y}(s)$ denote the non-extinction probabilities of
 ${\cal X}(s)$ and ${\cal Y}(s)$ respectively.
  In particular, for the general model of an inhomogeneous digraph,
  (\ref{result}) reads
 as follows
\begin{equation}\label{result-??}
N_1({\cal D}_n)=\rho_{{\cal X}{\cal Y}}\, n+o_P(n)
\qquad
{\text{as}}
\qquad
n\to \infty.
\end{equation}
In order to establish (\ref{result-??}) we need to impose  further conditions on the
kernel $\kappa$, like those in \cite{BJR}.
Namely, we need to assume that
 the
kernel $\kappa$ is  irreducible $Q\times Q$ almost everywhere (see {Remark} after Theorem \ref{T1}). That is, for any measurable $A\subset S$ with
$Q(A)\not=1$ or $0$,
the identity
$Q\times Q\bigl(\bigl\{(s,t)\in A\times (S\setminus A):\, \kappa(s,t)\not=0\bigr\}\bigr)=0$
 implies $Q(A)=0$ or $Q(S\setminus A)=0$,
see \cite{BJR}.
In addition we assume that $\kappa$ is continuous almost everywhere on $(S\times S, Q\times Q)$,
and the number of arcs in ${\cal D}_n$, denoted by $|{\cal D}_n|$, satisfies
\begin{displaymath}
n^{-1}\E|{\cal D}_n|\to \int\int_{S\times S}\kappa(s,t)Q(ds)Q(dt)<\infty
\end{displaymath}
 as $n\to\infty$. Note that  here we implicitly assume that $\kappa$ is integrable, i.e.,
 $\kappa\in L_1(S\times S,Q\times Q)$.

We will not give the proof of (\ref{result-??}). It can be obtained  from Theorem \ref{T1} via a finite
dimensional
approximation argument similar to that of the proof of Theorem 3.1 in \cite{BJR}.

\section{Proof}
In the proof we shall use ideas and techniques developed  in Karp \cite{Karp} and Bollob\'as, Janson and Riordan \cite{BJR}.

\begin{proof}[Proof of Theorem  \ref{T1}]
Given a vertex $v\in V$, let $X(v)$ denote the set of vertices that can be reached from $v$ via d-paths,
and let $Y(v)$ denote the set of starting points of d-paths ending at $v$,
\begin{displaymath}
X(v)=\{u\in V:\, v\rightsquigarrow u\},
\qquad
Y(v)=\{u\in V:\, u\rightsquigarrow v\}.
\end{displaymath}

Given  a function $\omega(n)$ such that $\omega(n)\to \infty$ and $\omega(n)=o(n)$ as $n\to\infty$,
we  say that $v\in V$ is {\it $x-$big} (respectively {\it $y-$big})
if $|X(v)|\ge \omega(n)$ (respectively, $|Y(v)|\ge \omega(v)$). The set of $x-$big (respectively $y-$big) vertices is
denoted $B_x=B_x(\omega)$ (respectively $B_y=B_y(\omega)$). We write
$B=B(\omega)=B_x(\omega)\cap B_y(\omega)$ for the set of vertices which are  $x-$big and $y-$big simultaneously.

We show that for any such function $\omega$,
\begin{equation}\label{BB-1}
n^{-1}|B(\omega)|-\rho_{{\cal X}{\cal Y}}=o_P(1)
\qquad
{\text{as}}
\quad
n\to\infty.
\end{equation}
For this purpose it suffices to show that for each $\omega$ we have
\begin{equation}\label{BB-2}
n^{-1}\E|B(\omega)|=\rho_{{\cal X}{\cal Y}}+o(1)
\end{equation}
and to establish (\ref{BB-1}) for at least one such  function, say  $\omega_0(n)=\ln n$.

Indeed, assume that (\ref{BB-2}) holds.
Let $\omega,\omega'$ be two such functions and $B,B'$ denote the correspond\-ing sets of large vertices.
For the size of the symmetric difference $B\bigtriangleup B'=(B\cup B')\setminus (B\cap B')$ equals
$|B(\omega\wedge\omega')|-|B(\omega\vee\omega')|\ge 0$, we obtain from (\ref{BB-2}) that $\E|B\bigtriangleup B'|=o(n)$.
It follows that $\bigl||B|-|B'|\bigr|\le |B\bigtriangleup B'|=o_P(n)$. In particular,  (\ref{BB-1}) holds for every $\omega$ whenever it is satisfied
by at least one such function $\omega$.

{\it Proof of (\ref{BB-2})}.
{\it Forward exploration.}
Given $v\in V$, we explore the set $X(v)$ as follows. Color all vertices blue. Color $v$ white
and put it in the list, which now contains a single white vertex $v$.
Then proceed recursively:\, choose a white vertex from the list, color it black,
reveal all { outgoing arcs emerging} 
from this vertex to blue vertices, color these vertices white and add
 {them} to the list.
Stop when we have collected at least $\omega(n)$ vertices in the list ({hence} $v\in B_x(\omega)$),
or there are no white vertices left in the list
( {hence} we have explored entire set $X(v)$ and $v\notin B_x(\omega)$). Let $X_{\omega}(v)$ denote the set of vertices
 collected in the list. Given $u,w\in X(v)$ we say that $u$ is an {\it f-child} (forward child) of $w$ if
 $\{w\rightarrow u \}\in {\cal D}$ and
 $u$ was blue when
 $w$ discovered it during the exploration process.
 {Since}  the last black vertex  adds to the list all its  blue neighbours (endpoints of
outgoing arcs), the size $|X_{\omega}(v)|$ can not exceed $\omega(n)$ by more than the out-degree
$\Delta_x$
of the last black vertex.

{\it Backward exploration.}
We perform the same exploration process starting at $v$ as above,
but
in the transposed digraph ${\cal D}^{*}$, which is obtained from ${\cal D}$ by reversing direction of arcs
($\{v\rightarrow~u\}\in~{\cal D}\Leftrightarrow\{u\rightarrow v\}\in {\cal D}^{*}$).
That is, now the search {for} neighbours propagates 
 {in the reverse}  direction of  {the}  arcs of ${\cal D}$.
  Let $Y_{\omega}(v)$ be the subset of $Y(v)$ obtained in at most $\omega(n)$
  steps of  {the} exploration.
Given  $u,w\in Y_{\omega}(v)$, the vertex $u$ is called a {\it  $b-$child} (backward child)
of $w$ if
$\{u\rightarrow w\}\in {\cal D}$
and $u$ was blue when
 $w$ discovered it. Again, we have $|Y_{\omega}(v)|\le \omega(n)+\Delta_y$, where now $\Delta_y$ is the
out-degree in ${\cal D}^{*}$ (in-degree in ${\cal D}$) of the vertex last
explored.

By a coupling of an exploration process with the approximating Galton-Watson process, Bollo\-b\'as, Janson and Riordan
\cite{BJR} showed
that the fraction of large vertices (the number $|B(\omega)|/n$) converges in probability to the  survival
probability of the Galton-Watson process, see Lemma 9.6 in \cite{BJR}.
Their results are stated  for (undirected) random graphs, but
several steps of their proof extend to random digraphs as well.
In particular, by a coupling of  the
forward exploration process with ${\cal X}$ (backward exploration process with ${\cal Y}$) we obtain as $n\to\infty$
uniformly in $v\in V$
\begin{eqnarray}\label{B-xy1}
&&
\PP(v\in B_x(\omega))=\rho_{\cal X}(s(v))+o(1),
\qquad
\,
\PP(v\in B_y(\omega))=\rho_{\cal Y}(s(v))+o(1),
\\
\label{B-xy2}
&&
\PP(\Delta_x>\ln n)=O(n^{-1}),
\qquad
\qquad
\qquad
\PP(\Delta_y>\ln n)=O(n^{-1}).
\end{eqnarray}

\bigskip

Let us show (\ref{BB-2})  for  $\omega(n)$ satisfying $\omega(n)\le \ln n$.
Introduce the events
${\cal A}_x(v)=\{|X_{\omega}(v)|\ge\omega(n)\}$ and ${\cal A}_y(v)=\{|Y_{\omega}(v)|\ge\omega(n)\}$.
Since  $|X_{\omega}(v)|\ge\omega(n) \Leftrightarrow v\in B_x(\omega)$ and
$|Y_{\omega}(v)|\ge\omega(n) \Leftrightarrow v\in~B_y(\omega)$ we have
$\PP(v\in B(\omega))=\PP({\cal A}_x(v)\cap{\cal A}_y(v))$. In view of (\ref{B-xy2})  with a high probability
each of the events ${\cal A}_x(v)$ and ${\cal A}_y(v)$ refer
to at most $\omega(n)+\ln n\le 2\ln n$ vertices. Therefore, we may expect that for large $n$ these events
are almost independent and we have (see (\ref{B-xy1}))
\begin{equation}\label{B-xy5}
\PP\bigl(v\in B(\omega)\bigr) = \PP({\cal A}_x(v)\cap{\cal A}_y(v))=\rho_{\cal Y}(s(v))\rho_{\cal X}(s(v))+o(1).
\end{equation}
We show that (\ref{B-xy5}) holds uniformly in $v\in V$. Then (\ref{BB-2})  follows from (\ref{Q-n}), (\ref{RO-1}) and
 (\ref{B-xy5}) via the identities
\begin{equation}\label{B-xy5+}
\E|B(\omega)|=\E\sum_{v\in V}{\mathbb I}_{\{v\in B(\omega)\}}=\sum_{v\in V}\PP\bigl(v\in B(\omega)\bigr).
\end{equation}
Let us prove (\ref{B-xy5}). Let ${\cal A}=\{X_{\omega}(v)\cap Y_{\omega}(v)=v\}$ denote the event that two exploration processes
after starting at $v$ do not meet each other in the first $\omega(n)$ steps of the exploration. Observe, that
uniformly in $v\in V$
\begin{equation}\label{B-xy3}
\PP({\cal A})=1-o(1)
\qquad
{\text{as}}
\qquad
n\to\infty.
\end{equation}
Indeed, assume that the set $X_{\omega}(v)$ is already constructed and now we construct the set $Y_{\omega}(v)$.
  Note that conditionally, given
 $X_{\omega}(v)$  {being} of size at most $\omega(n)+\ln n$, each black vertex of $Y_{\omega}(v)$ discovers at least one
b-child in $X_{\omega}(v)$ with probability at most $|X_{\omega}(v)|p_*n^{-1}\le (\omega(n~+~\ln n)p_*n^{-1}$, where
$p_*=\max_{i\le i,j\le k}p_{ij}$. 
Since there are at most $\omega(n)$ black vertices in $Y_{\omega}(v)$,
we conclude that on the event ${\cal D}_x=\{|X_{\omega}(v)|\le \omega(n)+\ln n\}$
the conditional proba\-bi\-lity
\begin{equation}\label{B-xy4}
\PP\bigl({\overline{\cal A}}\, \bigr|\, X_{\omega}(v)\bigr)\le \omega(n)(\omega(n)+\ln n) p^*n^{-1}.
\end{equation}
Here ${\overline {\cal A}}$ denotes the   {complementary event} to ${\cal A}$.
It follows from (\ref{B-xy4})
that  $\PP({\overline{\cal A}}\cap {\cal D}_x)=o(1)$.
The latter bound combined with (\ref{B-xy2}) shows (\ref{B-xy3}).

Now we are ready to show (\ref{B-xy5}). Assume again that $X_{\omega}(v)$  {has} already
 {been} constructed.  {Now we have} to  construct the set $Y_{\omega}(v)$.
 The vertices of $V'=V\setminus X_{\omega}(v)$ remain blue. In particular
for every $i$ the set $V'$ contains
 at least $n_i-(\omega(n)+\ln n)=n_i(1-o(1))$ blue vertices of type $s_i$. Conditionally on
the event ${\cal A}\cap{\cal D}_x$, the exploration of  $Y(v)$ (until we stop it after at most $\omega(n)$ steps)
stays  within the set $V'$ of
size $n(1-o(1))$. The second identity of (\ref{B-xy1}) applies to the conditional probability
$\PP\bigl({\cal A}_y(v)\,\bigr|\, X_{\omega}(v), {\cal A}\bigr)$ and yields
\begin{equation}\nonumber
\PP\bigl({\cal A}_y(v)\,\bigr|\, X_{\omega}(v), {\cal A}\bigr)= \rho_{\cal Y}(s(v))+o(1)
\end{equation}
uniformly in $X_{\omega}(v)$, satisfying the event ${\cal D}_x$.
Therefore, we have
\begin{eqnarray}\nonumber
\PP\bigl({\cal A}_y(v)\cap{\cal A}_x(v)\cap{\cal A}\cap{\cal D}_x\bigr)
&
=
&
\PP\bigl({\cal A}_y(v)\,\bigr|\,{\cal A}_x(v)\cap{\cal A}\cap{\cal D}_x\bigr)
\PP({\cal A}_x(v)\cap{\cal A}\cap{\cal D}_x)
\\
\nonumber
&
=
&
\rho_{\cal Y}(s(v))\, \PP({\cal A}_x(v)\cap{\cal A}\cap{\cal D}_x)+o(1).
\end{eqnarray}
In view of (\ref{B-xy2}), (\ref{B-xy3}) we can replace $\PP\bigl({\cal A}_y(v)\cap{\cal A}_x(v)\cap{\cal A}\cap{\cal D}_x\bigr)$
by $\PP({\cal A}_y(v)\cap{\cal A}_x(v))$ and $\PP({\cal A}_x(v)\cap{\cal A}\cap{\cal D}_x)$ by $\PP({\cal A}_x(v))$.
Therefore, we obtain $\PP({\cal A}_x(v)\cap{\cal A}_y(v))=\rho_{\cal Y}(s(v))\PP({\cal A}_x(v))+o(1)$. Finally, invoking (\ref{B-xy1})
we obtain(\ref{B-xy5}).
{Thus} we have proved (\ref{BB-2}) for $\omega(n)$ satisfying the extra condition $\omega(n)\le \ln n$.

Let us  {now} prove  (\ref{BB-2}) for arbitrary $\omega$ (satisfying $\omega(n)\to\infty$ and $\omega(n)=o(n)$ as $n\to\infty$).
Fix such an $\omega$. We apply (\ref{BB-2}) to $\omega'(n)=\omega(n) \wedge \ln(n)$
and invoking the inequality
$|B(\omega)|\le |B(\omega')|$,  {we} obtain the upper bound $\E|B(\omega)|\le n\rho_{{\cal X}{\cal Y}}+o(n)$.
The corresponding lower bound
\begin{equation}\label{BB-3}
\E|B(\omega)|\ge n\rho_{{\cal X}{\cal Y}}+o(n).
\end{equation}
follows from (\ref{B-xy5+}) and the inequalities, which hold uniformly in $v\in V$,
\begin{equation}\label{B-xy6}
\PP(v\in B(\omega))\ge \rho_{\cal X}(s(v))\rho_{\cal Y}(s(v))-o(1).
\end{equation}
To show that (\ref{B-xy6}) holds,
we  {first} perform a forward exploration starting at $v$ and obtain the set $X_{\omega}(v)$.  {Afterwards},
in the digraph induced by the vertex set $V^0:=\bigl(V\setminus X_{\omega}(v)\bigr)\cup\{v\}$
we perform a backward
exploration  starting at $v$. The set of  {the} thus discovered vertices is
denoted  {by}  $Y^0_{\omega}(v)$.
For  a set $V^0$ containing at least $n_i-(\omega(n)+\ln n)=n_i(1-o(1))$ vertices of every type $s_i\in S$,
the approximation of the distribution of the backward exploration process by the distribution of the Galton-Watson process ${\cal Y}(v)$
remains valid. That is, the second identity of (\ref{B-xy1}) applies to the conditional
probability $\PP\bigl(|Y^0_{\omega}(v)|\ge \omega(n)\,\bigr|\, X_{\omega}(v)\bigr)$.
We have
\begin{equation}\label{B-xy7}
\PP\bigl(|Y^0_{\omega}(v)|\ge \omega(n)\,\bigr|\, X_{\omega}(v)\bigr)=\rho_{\cal Y}(s(v))+o(1),
\end{equation}
uniformly $v$ and in $X_{\omega}(v)$ satisfying $|X_{\omega}(v)|\le \omega(n)+\ln n$.
Since $Y(v)$ contains $Y^0_{\omega}(v)$ as~a~subset, we obtain
\begin{displaymath}
\PP\bigl(|Y(v)|\ge \omega(n)\,\bigr|\, X_{\omega}(v)\bigr)\ge \rho_{\cal Y}(s(v))+o(1).
\end{displaymath}
The latter inequality in combination with the first identity of (\ref{B-xy1}) and the first bound of (\ref{B-xy2})
shows (\ref{B-xy6}). We now have arrived at  (\ref{BB-3}), thus completing the
proof of (\ref{BB-2}).
\vskip 0.05cm
{\it Proof of (\ref{BB-1}).} We prove (\ref{BB-1}) for a
 particular function $\omega_0(n)=\ln n$. Note that (\ref{BB-1}) follows from (\ref{BB-2}) and the bound
\begin{equation}\label{VB-1}
\Var|B(\omega_0)|=o(n^2)
\end{equation}
via  Chebyshev's inequality.
In addition, (\ref{VB-1}) follows from
the identities
\begin{eqnarray}\nonumber
&&
\Var |B(\omega_0)|=\E |B(\omega_0)|^2-(\E|B(\omega_0)|)^2,
\\
\nonumber
&&
\E |B(\omega_0)|(|B(\omega_0)|-1)
=
2\E\sum_{\{u,v\}\subset V}{\mathbb I}_{\{v\in B(\omega_0)\}}{\mathbb I}_{\{u\in B(\omega_0)\}}
\end{eqnarray}
combined with the identity, which holds uniformly  in $\{u,v\}\subset V$,
\begin{equation}\nonumber
\E {\mathbb I}_{\{v\in B(\omega_0)\}}{\mathbb I}_{\{u\in B(\omega_0)\}}
=
\rho_{\cal X}(s(v))\rho_{\cal Y}(s(v))\rho_{\cal X}(s(u))\rho_{\cal Y}(s(u))+o(1).
\end{equation}
The latter asymptotic identity is shown in much the same way as  (\ref{B-xy5})
above.
\vskip 0.1cm
{\it Proof of (\ref{result}).} Write $N_1=N_1({\cal D}_{{\mathbb P},{\overline n}})$.
Note that for every  $\omega(n)$ we have $N_1\le \omega(n)\vee|B(\omega)|$.
In combination with (\ref{BB-1}) this inequality implies the upper bound
\begin{equation}\label{BB-upper}
N_1n^{-1}\le \rho_{{\cal X}{\cal Y}}+o_P(1).
 \end{equation}
 Here and below for a sequence of random variables $\{Z_n\}$  we write $Z_n\le o_P(1)$
 (respectively   $Z_n\ge o_P(1)$) if for every $\delta>0$ we have
$\lim_n\PP(Z_n\le \delta)=1$ (respectively $\lim_n\PP(Z_n\ge -\delta)=1$).

\smallskip

For $\rho_{{\cal X}{\cal Y}}=0$ the result (\ref{result}) follows from (\ref{BB-upper}).
 For $\rho_{{\cal X}{\cal Y}}>0$ the result (\ref{result}) follows from  (\ref{BB-upper}) and the lower bound
\begin{equation}\label{BBB-0}
N_1n^{-1}\ge \rho_{{\cal X}{\cal Y}}+o_P(1).
 \end{equation}
In order to show this lower bound
we generate the digraph ${\cal D}$  in two steps. Firstly we generate
a digraph ${\cal D}'={\cal D}_{||p'_{ij}||, {\overline n}}$ and then, on the top of it,
we generate another digraph ${\cal D}''={\cal D}_{||p''_{ij}||, {\overline n}}$ independently of ${\cal D}'$.
Here the numbers $p'_{ij}$ and $p''_{ij}$ are defined by the equations
\begin{equation}\label{p'ij}
p'_{ij}=p_{ij}(1-\varepsilon),
\qquad
(1-p''_{ij}n^{-1})(1-p'_{ij}n^{-1})=1-p_{ij}n^{-1},
\qquad
1\le i,j\le k.
\end{equation}
Here $0<\varepsilon<1$ is fixed and we assume that $n$ is so large that all $p_{ij}n^{-1}<1$.
The union ${\cal D}'\cup {\cal D}''$ is the digraph on the vertex set $V$ such that,  for each $(u,v)\in [V]^2$,
we have $\{v\rightarrow u\}\in {\cal D}'\cup {\cal D}''$
whenever
$\{v\rightarrow u\}$
is present in at least one of the digraphs ${\cal D}'$ and ${\cal D}''$.
Note that, by the second equation of (\ref{p'ij}),
the random digraphs ${\cal D}'\cup {\cal D}''$ and ${\cal D}$ have the same probability distribution.

Let ${\cal X}'=\{{\cal X}'(s),\, s\in S\}$ and ${\cal Y}'=\{{\cal Y}'(s),\, s\in S\}$ be the multi-type
Galton-Watson processes with Poisson offspring distributions that
approximate  the forward and backward explorations of neighbourhoods of vertices in ${\cal D}'$. They are defined
in the same way as ${\cal X}$ and ${\cal Y}$
above, but with respect to the matrix $||p'_{ij}||$.
Let $\rho_{[\varepsilon]}=\rho_{{\cal X}'{\cal Y}'}$ be defined by (\ref{RO-1}). One can show (e.g.,
by coupling of ${\cal X}(s)$ with ${\cal X}'(s)$ and ${\cal Y}(s)$ with ${\cal Y}'(s)$) that
$\rho_{\cal X}'(s)\to\rho_{\cal X}(s)$ and $\rho_{\cal Y}'(s)\to\rho_{\cal Y}(s)$ as $\varepsilon \downarrow 0$.
In particular, we have
\begin{equation}\label{RO-2}
\lim_{\varepsilon\to 0}\rho_{[\varepsilon]}=\rho_{{\cal X}{\cal Y}}.
\end{equation}

We are now ready to prove (\ref{BBB-0}).
Fix the function $\omega_1(n)=n/\ln n$. Given $v\in V$,  let $X'_{\omega_1}(v)$ and $Y'_{\omega_1}(v)$ denote the neighbourhoods
of $v$ discovered
by the forward and backward explorations performed in ${\cal D}'$.
Let $B'(\omega_1)$ denote the set of large vertices of ${\cal D}'$. From (\ref{BB-1}) we obtain
\begin{equation}\label{BBB-1}
n^{-1}|B'(\omega_1)|-\rho_{[\varepsilon]}=o_P(1).
\end{equation}
We shall show that with  high probability every pair of vertices from $B'(\omega_1)$ communicate in ${\cal D}$.
This will imply that $N_1$ is at least as large as $|B'(\omega_1)|$, and, as a consequence, we then obtain the lower
bound (\ref{lower-bound}).

We generate the digraph ${\cal D}''$ in $k$ steps so that ${\cal D}''={\cal D}_1\cup\cdots\cup {\cal D}_k$.
Here ${\cal D}_i$, $1\le i\le k$ are independent copies of ${\cal D}_{||p^*_{ij}||,{\overline n}}$, where
the matrix  $||p^*_{ij}||$ is defined by  the equations
\linebreak
$(1-p^*_{ij}n^{-1})^k=1-p''_{ij}n^{-1}$, $1\le i,j\le k$.
 Note that for  $\{s_i\rightarrow s_j\}\in D_P$ identities (\ref{p'ij}) imply
\begin{equation}\label{P-ij-2}
p^*_{ij}\ge k^{-1}\varepsilon p_{ij} \ge p_0,
\qquad
p_0:=k^{-1}\varepsilon \min\{p_{ij}:\, p_{ij}>0\}>0.
\end{equation}

For $i=1,\dots, k$ denote
$X^i(v)=\{u\in V:\, \{v \rightsquigarrow u\}\in {\cal D}'\cup{\cal D}_1\cup\cdots\cup{\cal D}_i\}$ and write
$V(s_i)=\{v\in V:\, s(v)=s_i\}$.
We shall show that with  high probability  every  set $X^{k-1}(v)$, $v\in B'(\omega_1)$, contains
at least
$\Theta(n(\ln n)^{-1})$ vertices of each type. More precisely,   the event
\begin{displaymath}
{\cal H}=\cap_{v\in B'(\omega_1)}\cap_{1\le i\le k}\bigl\{|X^{k-1}(v)\cap V(s_i)|\ge n(\ln n)^{-1}\varkappa\bigr\},
\end{displaymath}
where
$\varkappa:=k^{-1}(p_0/4)^{k-1}q_1\times\cdots\times q_k$, has probability
\begin{equation}\label{H-1}
\PP\bigl({\cal H}\bigr)=1-o(1)
\qquad
{\text{as}}
\qquad
n\to\infty.
\end{equation}
Observe, that the event ${\cal H}$ depends on the random graph
${\cal D}^{\star}={\cal D}'\cup{\cal D}_1\cdots\cup{\cal D}_{k-1}$ and is independent of ${\cal D}_k$.
We show that on the event ${\cal H}$ the conditional probability, which now refers to  ${\cal D}_k$,
satisfies
\begin{equation}\label{N1-x}
\PP\Bigl(N_1\ge |B'(\omega_1)|\,\Bigr|\, {\cal D}^{\star}\Bigr)=1-o(1).
\end{equation}
This bound in combination with (\ref{H-1}) and (\ref{BBB-1}) implies the lower bound
\begin{equation}\label{lower-bound}
N_1n^{-1}\ge \rho_{[\varepsilon]}+o_P(1).
\end{equation}
Letting $\varepsilon \downarrow 0$ from (\ref{RO-2}) we obtain (\ref{BBB-0}).
We complete the proof of (\ref{result}) by first  showing (\ref{N1-x}) and then (\ref{H-1}).
\vskip 0.1cm

{\it Proof of (\ref{N1-x}).} Given ${\cal D}^{\star}$,  define the events
\begin{displaymath}
{\cal A}_{uv}=
\bigl\{
\exists
\ \
x\in X^{k-1}(u),
\ \
\exists
\ \
y\in Y'_{\omega_1}(v)
\ \
{\text{such that}}
\ \
\{ x \rightarrow y\}\in {\cal D}_{k}
\bigr\},
\qquad
u,v\in B'(\omega_1).
\end{displaymath}
Let ${\overline A}_{uv}$ denote the event complement to ${\cal A}_{uv}$. Introduce the sum
\begin{displaymath}
S
=
\sum_{\{u,v\}\subset B'(\omega_1)}
 {\mathbb I}_{{\overline {\cal A}}_{uv}}
 \end{displaymath}
which (given ${\cal D}^{\star}$) is at least as large as the number of pairs $\{u,v\}\in B'(\omega_1)$ that do not communicate in ${\cal D}$.
We claim that on the event ${\cal H}$, we have $S=0$ with  high (conditional given ${\cal D}^{\star}$) probability.
Indeed, the largest of  the sets $Y'_{\omega_1}(v)\cap V(s_i)$, $1\le i\le k$  is of size at least $\omega_1(n)/k$.
Assume it is the r-th set $Y_r:=Y'_{\omega_1}(v)\cap V(s_r)$.  Since $D_P$ is strongly connected,
$\{s_j\rightarrow s_r\}\in D_P$ for some $s_j$.   Given the event ${\cal H}$, the set  $X_j:=X^{k-1}(v)\cap V(s_j)$
is of size  at least  $\Theta(n(\ln n)^{-1})$.
Therefore, we have
\begin{equation}\label{u-v}
\PP\bigl({\overline {\cal A}}_{uv}\,\bigr|\,{\cal D}^{\star}\bigr)
\le
(1-p^*_{jr}n^{-1})^{|X_j|\times|Y_r|}
\le
(1-p_0n^{-1})^{|X_j|\times|Y_r|}\le c'n^{-4}.
\end{equation}
Here $c'>0$ denotes a  constant depending only  on $||p_{ij}||$ and $\varepsilon$.
It follows from (\ref{u-v})
by
Chebyshev's inequality  that given the event ${\cal H}$ we have
\begin{displaymath}
1-\PP(S=0|{\cal D}^{\star})=\PP(S\ge 1|{\cal D}^{\star})\le\E(S|{\cal D}^{\star})\le c'n^{-2}.
\end{displaymath}
Since $N_1\ge |B'(\omega_1)|-S$, the latter bound implies (\ref{N1-x}).

\smallskip

{\it Proof of (\ref{H-1}).} Let $\PP'(\cdot)=\PP(\cdot|{\cal D}')$ denote the conditional probability
given ${\cal D}'$. In order to prove (\ref{H-1}), we show that, for each $v\in B'(\omega_1)$ and $1\le i\le k$,
\begin{equation}\label{H-2}
\PP'\bigl(|X^{k-1}(v)\cap V(s_i)|< n(\ln n)^{-1}\varkappa\bigr)=O(n^{-4}).
\end{equation}

Fix $v\in B'(\omega_1)$  and $i$.
The largest of   sets $X'_{\omega_1}(v)\cap V(s_j)$, $1\le j\le k$  is of size at least $\omega_1(n)/k$.
Assume it is the r-th set $X^0:=X'_{\omega_1}(v)\cap V(s_r)$. Since $D_P$ is strongly connected we find a~shortest path
\begin{equation}\label{takas}
s_r=t_0\rightarrow t_1\rightarrow \cdots \rightarrow t_j=s_i
\end{equation}
 in $D_P$. Denote
$\varkappa_m=k^{-1}(p_0/4)^{m}Q(t_1)\cdots Q(t_m)$.
We claim that, for $1\le m\le j$,
\begin{equation}\label{H-3}
\PP'\bigl(|X^{m}(v)\cap V(t_m)|< n(\ln n)^{-1}\varkappa_m\bigr)=O(n^{-4}).
\end{equation}
Note that,  in view of the inclusions $X'_{\omega_1}(v)=X^0(v)\subset X^{1}(v)\subset\cdots\subset X^{k-1}(v)$
and  {the} inequalities $\varkappa\le \varkappa_i$, the bound
(\ref{H-3}) (for $m=j$) implies (\ref{H-2}). We show (\ref{H-3}), for $m=1,2,\dots$. Let $m=1$. Let $X^1$ denote the
set of endpoints in $V(t_1)$ of arcs in ${\cal D}_1$ that start at vertices in $X^0$. We have
\begin{displaymath}
|X^{1}(v)\cap V(t_1)|\ge |X^1|=\sum_{w\in V(t_1)}{\mathbb I}_{\{w\in X^1\}}
\end{displaymath}
The right-hand sum has binomial distribution $Bi(|V(t_1)|, p_1)$ with $|V(t_1)|=nQ(t_1)+o(n)$ trials
and  success probability
\begin{displaymath}
p_1\ge 1-(1-p_0n^{-1})^{|X^0|}\ge 1-(1-p_0n^{-1})^{\omega_1(n)/k}.
\end{displaymath}
 Indeed, $p_0n^{-1}$ is the smallest probability of arcs in ${\cal D}_1$, and
there are at least $|X^0|\ge \omega_1(n)/k$ vertices  in $X^0(v)$ that "try" to send an arc to a given  vertex $w\in V(t_1)$.
A simple analysis show that $p_1=\Theta(ln^{-1}(n))$. In particular, for large $n$ we have
$p_1\ge p_0/(2k\ln n)$. Therefore, we obtain  $\E|X^1|\ge 2\varkappa_1n(\ln n)^{-1}(1+o(1))$.
Now, an application of Chernoff bound (see, e.g., Janson, {\L}uczak and Ruci\'nski \cite{janson2001})
to the binomial random variable $|X^1|$ shows (\ref{H-3}) for $m=1$.
 Next, given the event  $\{|X^{1}(v)\cap V(t_1)|\ge n(\ln n)^{-1}\varkappa_1\}$, which is independent of
${\cal D}_2$, we show
\begin{displaymath}
\PP\Bigl(|X^{2}(v)\cap V(t_2)|< n(\ln n)^{-1}\varkappa_2\, \Bigr| \, {\cal D}',{\cal D}_1\Bigr)=O(n^{-4}).
\end{displaymath}
in much the same way. That is, we show that number of different  endpoints in $V(t_2)$ of arcs in ${\cal D}_2$
that start
in $X^{1}(v)\cap V(t_1)$ is strongly concentrated around $2\varkappa_2n(\ln n)^{-1}$. In this way,
proceeding along the path
(\ref{takas}) until the last endpoint $s_j$ and using arcs from ${\cal D}_1$, ${\cal D}_2$ etc.  for
the successive steps
respectively,
we arrive at (\ref{H-3}).

In the very last step of the proof we show $N_2=o_P(n)$.
Indeed, this bound follows immediately from (\ref{BB-1}), (\ref{BBB-0}) via the simple
inequality $N_1+N_2\le 2\omega(n)+|B(\omega)|$.
\end{proof}

{\bf Acknowledgement.} Mindaugas Bloznelis  {would like to thank} Colin Cooper for valuable and  inspiring  discussion.

\end{document}